\documentclass[12pt]{amsart}

\usepackage{amsmath,amssymb,amsthm,amsfonts,mathtools,fullpage}
\usepackage[hidelinks]{hyperref}

\newtheorem{Theorem}{Theorem}[section]
\newtheorem{Lemma}[Theorem]{Lemma}
\newtheorem{Proposition}[Theorem]{Proposition}
\newtheorem{Corollary}[Theorem]{Corollary}
\newtheorem{Conjecture}[Theorem]{Conjecture}
\newtheorem{Question}[Theorem]{Question}

\theoremstyle{definition}

\newcommand{\ZZ}{\mathbb Z}
\newcommand{\RR}{\mathbb R}
\newcommand{\conv}{\operatorname{conv}}

\newcommand{\DD}{\Delta}
\newcommand{\PP}{\mathcal P}

\title[Magic positivity for pseudo-symmetric Fano duals]
{Magic positivity for polar duals of pseudo-symmetric smooth Fano polytopes}

\author{Hidefumi Ohsugi}
\address{Hidefumi Ohsugi,
Department of Mathematical Sciences,
School of Science,
Kwansei Gakuin University,
Sanda, Hyogo 669-1330, Japan}
\email{ohsugi@kwansei.ac.jp}

\author{Akiyoshi Tsuchiya}
\address{Akiyoshi Tsuchiya,
Department of Information Science,
Faculty of Science,
Toho University,
Chiba, Japan}
\email{akiyoshi@is.sci.toho-u.ac.jp}



\begin{document}

\begin{abstract}
Motivated by Gal's conjecture, Ferroni and Higashitani conjectured that 
the $h^*$-polynomial of any Gorenstein lattice polytope admitting a quadratic triangulation is
$\gamma$-positive.  
Together with conjectures predicting
quadratic properties of toric ideals of smooth lattice polytopes, this
suggests that smooth Gorenstein lattice polytopes should have
$\gamma$-positive $h^*$-polynomials.
In this paper, we prove that the Ehrhart polynomial of the polar dual
of every pseudo-symmetric simplicial reflexive polytope is magic positive.
For pseudo-symmetric smooth Fano polytopes, we prove the stronger
statement that all magic coefficients are strictly positive.
Consequently, for every pseudo-symmetric simplicial reflexive polytope,
its polar dual is Ehrhart positive and has a real-rooted and
$\gamma$-positive $h^*$-polynomial.
We also prove that the Ehrhart polynomial of the polar dual of the
symmetric edge polytope of every cycle is magic positive.  This gives an
affirmative answer to a question of Konoike, who had previously proved
partial positivity results for this family.
\end{abstract}

\maketitle

\section{Introduction}

Let $P\subset\RR^d$ be a $d$-dimensional lattice polytope.  Its Ehrhart
polynomial is
\[
L_P(k)=|kP\cap\ZZ^d|,
\]
and its Ehrhart series is
\[
1+\sum_{k\ge1}L_P(k)t^k
=
\frac{h^*(P;t)}{(1-t)^{d+1}}.
\]
We refer to \cite{BeckRobins} for the basic theory of Ehrhart polynomials
and Ehrhart series.  The numerator $h^*(P;t)$ is the $h^*$-polynomial of
$P$, whose coefficients are nonnegative integers by Stanley's theorem
\cite{StanleyDecompositions}.  If $P$ is Gorenstein, then $h^*(P;t)$ is
palindromic.  Writing $s=\deg h^*(P;t)$, it therefore has a unique
expansion
\[
h^*(P;t)
=
\sum_{i=0}^{\lfloor s/2\rfloor}
\gamma_i t^i(1+t)^{s-2i}.
\]
The polynomial $h^*(P;t)$ is called \emph{$\gamma$-positive} if
$\gamma_i\ge0$ for every $i$.

One of the central positivity conjectures in algebraic and geometric
combinatorics is Gal's conjecture \cite{Gal}, which predicts
$\gamma$-positivity of the $h$-polynomial of every flag simplicial sphere.
Motivated by this conjecture, Ferroni and Higashitani proposed an
Ehrhart-theoretic analogue.  They conjectured that a Gorenstein lattice
polytope admitting a quadratic triangulation has a $\gamma$-positive
$h^*$-polynomial
\cite[Conjecture~3.28]{FerroniHigashitani}.  Here a quadratic triangulation
means a regular unimodular flag triangulation.

Another long-standing problem concerns the toric ideals of smooth
lattice polytopes. The B{\o}gvad conjecture predicts that the toric ideal
of every smooth lattice polytope is generated by quadratic binomials.
A stronger version asks
for a quadratic Gr\"obner basis, while a squarefree quadratic initial ideal
corresponds to a quadratic triangulation via the standard correspondence
between initial ideals of toric ideals and regular triangulations; see
\cite[Chapter~8]{SturmfelsGBCP}.  Although these statements are logically
distinct, together with the conjecture of Ferroni and Higashitani they
suggest the following positivity problem.

\begin{Conjecture}\label{conj:smooth-gorenstein}
Every smooth Gorenstein lattice polytope has a $\gamma$-positive
$h^*$-polynomial.
\end{Conjecture}

For background and classification results on smooth Gorenstein polytopes,
see \cite{LorenzNill}.

The purpose of this paper is to establish a stronger positivity property
for the polar duals of pseudo-symmetric simplicial reflexive polytopes.
A reflexive polytope is called \emph{pseudo-symmetric} if it has a pair
of opposite facets.

Nill classified pseudo-symmetric simplicial reflexive polytopes in
arbitrary dimension \cite{NillPseudoSymmetric}.  His classification
decomposes every such polytope into a free sum of a centrally symmetric
reflexive crosspolytope, even-dimensional del Pezzo polytopes, and
pseudo-del Pezzo polytopes.  The smooth Fano members form an important
subclass, classified earlier by Ewald
\cite{Ewald,EwaldBook}.
The smooth Fano subclass has been particularly tractable in several
aspects of toric geometry.  
For
instance, Haase used these polytopes and their polar duals to compute
stringy Hodge numbers of generic hypersurfaces in the associated toric
varieties and to study related questions concerning mirror symmetry and
crepant resolutions \cite{HaaseThesis}.  The toric Fano varieties
associated with the centrally symmetric members of this family have also
been studied in connection with Tian's $\alpha$-invariant and
K\"ahler--Einstein metrics \cite{BatyrevSelivanova,SongAlpha}.  
On the Ehrhart side, in our previous work
\cite[Theorem~5.8]{OhsugiTsuchiyaLocally}, we stated that the
$h^*$-polynomial of every pseudo-symmetric simplicial reflexive polytope
is $\gamma$-positive.
The proof there invokes Nill's decomposition, but the centrally
symmetric reflexive crosspolytope factor was inadvertently treated as
a standard crosspolytope.
We supply the missing argument in Appendix~\ref{app:previous-proof}.
The present paper instead studies the polar duals and proves the
stronger property of magic positivity.
These structural, geometric,
and enumerative properties make the polar duals of pseudo-symmetric
smooth Fano polytopes a natural testing ground for
Conjecture~\ref{conj:smooth-gorenstein}.

We recall the relevant geometric terminology and the correspondence
between smooth Fano polytopes and smooth reflexive polytopes in
Section~\ref{sec:closure}.  Our first main result is the following.

\begin{Theorem}\label{thm:pseudo-symmetric-main}
Let $P$ be a pseudo-symmetric simplicial reflexive polytope.
Then the Ehrhart polynomial $L_{P^\vee}(k)$ is magic positive.
If $P$ is a smooth Fano polytope, then all magic coefficients of
$L_{P^\vee}(k)$ are strictly positive.
\end{Theorem}

Here, following Ferroni and Higashitani
\cite[Section~4.5]{FerroniHigashitani}, a polynomial $L(k)$ of degree $d$ has a unique expansion
\[
L(k)=\sum_{i=0}^d a_i k^i(k+1)^{d-i}.
\]
The numbers $a_0,\ldots,a_d$ are called the \emph{magic coefficients}
of $L$, and $L$ is called \emph{magic positive} if $a_i\ge0$ for every $i$.
We call
\[
M_L(y)
:=
(1-y)^dL\left(\frac{y}{1-y}\right)
=
\sum_{i=0}^d a_i y^i
\]
the \emph{magic transform} of $L$.
Thus $L$ is magic positive if and only if $M_L(y)$ has nonnegative
coefficients.
A lattice polytope $P$ is called \emph{Ehrhart positive} if all
coefficients of $L_P(k)$ in the standard monomial basis are positive.
If $L(k)=L_P(k)$, then $a_0=L_P(0)=1$, so magic positivity of $L_P(k)$
implies Ehrhart positivity of $P$.  It also implies real-rootedness of the
corresponding $h^*$-polynomial by the P\'olya-frequency-preserving
transformation studied by Br\"and\'en; see
\cite{BrandenPF} and \cite[Theorem~4.19]{FerroniHigashitani}.  For a
reflexive polytope, real-rootedness in turn implies $\gamma$-positivity.

Recent work has established magic positivity for other families, including
Pitman--Stanley polytopes and broad classes of partial permutohedra
\cite{AvilaFerroniMorales,LiuZhangPartialPermutohedra}.

We therefore obtain the following consequences.

\begin{Corollary}\label{cor:fano-consequences}
Let $P$ be a pseudo-symmetric simplicial reflexive polytope.
Then $P^\vee$ is Ehrhart positive, and $h^*(P^\vee;t)$ is
real-rooted and $\gamma$-positive.
\end{Corollary}
In particular, if $P$ is smooth Fano, then $P^\vee$ is a smooth
reflexive, and hence smooth Gorenstein, lattice polytope.
Thus Theorem~\ref{thm:pseudo-symmetric-main} verifies
Conjecture~\ref{conj:smooth-gorenstein} for the polar duals of
pseudo-symmetric smooth Fano polytopes.

Haase proved that the polar dual of every pseudo-symmetric smooth Fano
polytope admits a quadratic triangulation
\cite[Corollaries~IV.2.6 and~IV.2.10]{HaaseThesis}.
Thus these polytopes satisfy the triangulation hypothesis in the
conjecture of Ferroni and Higashitani.
On the other hand, Ferroni constructed Gorenstein lattice polytopes
admitting quadratic triangulations whose $h^*$-polynomials are not
log-concave \cite[Theorem~1.4 and Section~6]{FerroniShenanigans}.
Since magic positivity implies real-rootedness of the corresponding
$h^*$-polynomial, whose nonnegative coefficients are then log-concave
by Newton's inequalities, these examples show that the
quadratic-triangulation hypothesis alone does not imply magic
positivity, even in the Gorenstein case.

A second application concerns symmetric edge polytopes, introduced and
studied in \cite{MatsuiEtAlGraphs}.  For a graph $G$ on
$[n]:=\{1,2,\dots,n\}$, its symmetric edge polytope is
\[
\PP_G
=
\conv\{\pm(e_i-e_j):\{i,j\}\in E(G)\},
\]
where $e_i$ is the $i$-th standard unit vector of $\RR^n$.
For a connected graph $G$, we regard $\PP_G$ as a full-dimensional
lattice polytope in
$H=\{x\in\RR^n:x_1+\cdots+x_n=0\}$, equipped with the lattice
$H\cap\ZZ^n$, and take its polar dual with respect to the dual lattice.
Konoike studied magic positivity of the polar duals associated with
cycles.  For the cycle $C_n$, he obtained an explicit Ehrhart formula
\cite[Proposition~4.6]{KonoikeCycles}, proved positivity of the first
three and the last three magic coefficients
\cite[Theorem~1.4]{KonoikeCycles}, and asked whether all magic
coefficients are nonnegative
\cite[Question~4.5]{KonoikeCycles}.  We answer this question in a
stronger form.

\begin{Theorem}\label{thm:cycles}
For every $n\ge3$, all magic coefficients of
$L_{\PP_{C_n}^\vee}(k)$ are strictly positive.
\end{Theorem}

\begin{Corollary}\label{cor:cycle-consequences}
For every $n\ge3$, the polytope $\PP_{C_n}^\vee$ is Ehrhart positive,
and its $h^*$-polynomial is real-rooted and $\gamma$-positive.
\end{Corollary}

The strict positivity results for the del Pezzo and cycle cases are
based on a positivity result for divided differences, while the
pseudo-del Pezzo case follows from an averaging identity.
The general pseudo-symmetric simplicial reflexive case is then obtained
from Nill's classification and the binary-code description of centrally
symmetric reflexive crosspolytopes, while Ewald's classification yields
strict positivity in the smooth Fano case.

The paper is organized as follows.  Section~\ref{sec:closure} collects
the preliminaries and introduces the family of polytopes underlying
the proofs.  Section~\ref{sec:dd} recalls the divided-difference
identities used in the proof.  Section~\ref{sec:local-expansion} develops
the local decomposition of the Ehrhart and magic coefficients.
Section~\ref{sec:plus} establishes the positivity of the local terms.
Finally, Section~\ref{sec:applications} proves the main result for
pseudo-symmetric simplicial reflexive polytopes and applies these results
to polar duals of symmetric edge polytopes of cycles.

\subsection*{Acknowledgment}
This work was supported by JSPS KAKENHI Grant Numbers
JP22K13890, JP24K00534, and JP26K00618.

\section{Preliminaries and reduction}\label{sec:closure}

We refer to \cite{CoxLittleSchenck} for background on lattice polytopes
and toric varieties.  We first recall the terminology used throughout
the paper. Let $M\simeq
\ZZ^d$ be a lattice, let $N=\operatorname{Hom}(M,\ZZ)$ be its dual lattice,
and write $M_{\RR}=M\otimes_{\ZZ}\RR$ and $N_{\RR}=N\otimes_{\ZZ}\RR$.
For a full-dimensional polytope $P\subset M_{\RR}$ containing the origin in
its interior, its \emph{polar dual} is
\[
P^\vee
=
\{y\in N_{\RR}:\langle y,x\rangle\le1\text{ for every }x\in P\}.
\]
A lattice polytope $P\subset M_{\RR}$ is called \emph{reflexive} if the
origin lies in its interior and $P^\vee$ is again a lattice polytope.
A full-dimensional lattice polytope containing the origin in its
interior is called \emph{pseudo-symmetric} if it has a pair of
opposite facets.
A lattice polytope $P$ is called \emph{Gorenstein} if $rP-m$ is reflexive
for some positive integer $r$ and some lattice point $m\in M$.
The integer $r$ is called the \emph{index} of $P$.
A lattice polytope is Gorenstein if and only if its $h^*$-polynomial is
palindromic \cite{HibiDual,FerroniHigashitani}.

A $d$-dimensional lattice polytope $P\subset M_{\RR}$ is called
\emph{smooth} if, at every vertex $v$ of $P$, the primitive lattice vectors
along the $d$ edges emanating from $v$ form a basis of $M$.  A
full-dimensional lattice polytope $R\subset N_{\RR}$ containing the origin
in its interior is called a \emph{Fano polytope} if all its vertices are
primitive lattice points.  A Fano polytope is called \emph{smooth Fano} if
the vertices of every facet form a basis of $N$. 
We use the conventions of
\cite{Ewald,EwaldBook,HaaseThesis}.

A lattice polytope is smooth if and only if the projective toric
variety defined by its normal fan is smooth
\cite[Theorem~2.4.3]{CoxLittleSchenck}.
Reflexive polytopes are the anticanonical polytopes of Gorenstein
toric Fano varieties
\cite[Theorem~8.3.4]{CoxLittleSchenck}.

Polar duality identifies smooth Fano polytopes with smooth reflexive
polytopes.
Indeed, if the vertices $v_1,\ldots,v_d$ of a facet of $R$ form a
lattice basis, then the corresponding vertex of $R^\vee$ is
$v_1^*+\cdots+v_d^*$ and its primitive edge directions are
$-v_1^*,\ldots,-v_d^*$, where $v_1^*,\ldots,v_d^*$ is the dual basis.
The converse follows by considering the primitive facet normals
of a smooth reflexive polytope.

For full-dimensional polytopes $P\subset\RR^d$ and $Q\subset\RR^e$
containing the origin in their interiors, define their \emph{free sum} by
\[
P\oplus Q
=
\conv\bigl((P\times\{0\})\cup(\{0\}\times Q)\bigr).
\]
We will use the following elementary properties.

\begin{Lemma}\label{lem:closure}
Let $P$ and $Q$ be lattice polytopes.
\begin{enumerate}
\item If $P$ and $Q$ contain the origin in their interiors, then
\[
(P\oplus Q)^\vee=P^\vee\times Q^\vee.
\]
\item Their Ehrhart polynomials satisfy
\[
L_{P\times Q}(k)=L_P(k)L_Q(k).
\]
\item If $L_P(k)$ and $L_Q(k)$ are magic positive, then so is
$L_{P\times Q}(k)$.
If all magic coefficients of both factors are strictly positive,
then the same holds for their product.
\end{enumerate}
\end{Lemma}

For $m\ge1$, let
\[
V_{2m}
=
\conv\left(
\pm e_1,\ldots,\pm e_{2m},
\pm(e_1+\cdots+e_{2m})
\right)
\]
be the del Pezzo polytope, and let
\[
\widetilde V_{2m}
=
\conv\left(
\pm e_1,\ldots,\pm e_{2m},
-(e_1+\cdots+e_{2m})
\right)
\]
be the pseudo-del Pezzo polytope. 

We use Nill's classification of pseudo-symmetric simplicial reflexive
polytopes \cite[Theorem~0.1]{NillPseudoSymmetric}.

\begin{Proposition}[Nill]\label{prop:classification}
Every pseudo-symmetric simplicial reflexive polytope is lattice
equivalent to a free sum
\[
R
\oplus
V_{2m_1}\oplus\cdots\oplus V_{2m_r}
\oplus
\widetilde V_{2n_1}\oplus\cdots\oplus\widetilde V_{2n_s},
\]
where $r,s\ge0$, all $m_i,n_j\ge1$, and $R$ is a centrally symmetric
reflexive crosspolytope.  The factor $R$ may be omitted.
\end{Proposition}

For the smooth Fano subclass, Ewald's classification gives the following
more restrictive decomposition \cite{Ewald,EwaldBook};
see also \cite[Theorem~IV.2.1]{HaaseThesis}.

\begin{Proposition}[Ewald]\label{prop:smooth-classification}
Every pseudo-symmetric smooth Fano polytope is lattice equivalent to a
free sum of copies of $[-1,1]$, del Pezzo polytopes, and pseudo-del Pezzo
polytopes.
\end{Proposition}

The interval $[-1,1]$ is self-dual and has Ehrhart polynomial
\[
2k+1=k+(k+1),
\]
which is magic positive.

\subsection{Centrally symmetric reflexive crosspolytopes}
\label{sec:crosspolytope}
Let
\[
R=\conv(\pm v_1,\ldots,\pm v_d)
\]
be a $d$-dimensional centrally symmetric reflexive crosspolytope
in a lattice $N$.  Since $v_1,\ldots,v_d$ form an $\RR$-basis of
$N_{\RR}$, let
\[
\varphi:N_{\RR}\longrightarrow\RR^d
\]
be the linear isomorphism defined by $\varphi(v_i)=e_i$, and put
\[
\Lambda:=\varphi(N).
\]
Then
\[
\ZZ^d\subseteq\Lambda\subseteq\RR^d,
\qquad
R=\conv(\pm e_1,\ldots,\pm e_d).
\]
Thus
\[
R^\vee=[-1,1]^d
\]
with respect to the dual lattice $\Lambda^\vee$.

Since $R$ is reflexive, every vertex of $R^\vee$, that is, every
sign vector $\varepsilon\in\{\pm1\}^d$, belongs to $\Lambda^\vee$.
Hence, for every $x=(x_1,\ldots,x_d)\in\Lambda$,
\[
\varepsilon_1x_1+\cdots+\varepsilon_dx_d\in\ZZ
\qquad
(\varepsilon\in\{\pm1\}^d).
\]
Comparing two sign vectors which differ only in the $i$-th coordinate
gives
\[
2x_i\in\ZZ
\qquad(1\le i\le d),
\]
while taking $\varepsilon=(1,\ldots,1)$ gives
\[
x_1+\cdots+x_d\in\ZZ.
\]
In particular,
\[
\ZZ^d\subseteq\Lambda\subseteq\frac12\ZZ^d.
\]

Define
\[
\pi:\Lambda\longrightarrow\mathbb F_2^d,
\qquad
\pi(x)
=
(2x_1,\ldots,2x_d)\pmod2.
\]
Its kernel is $\ZZ^d$, and hence
\[
C:=\pi(\Lambda)\cong\Lambda/\ZZ^d
\]
is a binary linear code.
For $c\in\mathbb F_2^d$, let $\operatorname{wt}(c)$ denote its
Hamming weight.
Since $x_1+\cdots+x_d\in\ZZ$ for every $x\in\Lambda$,
every codeword of $C$ has even Hamming weight.
Thus $C$ is an even binary linear code.

Let
\[
D:=C^\perp
=
\left\{
b\in\mathbb F_2^d:
\sum_{j=1}^d b_jc_j=0 \text{ in }\mathbb F_2
\text{ for every }c\in C
\right\}
\]
be the dual code of $C$.
For $0\le i\le d$, let
\[
A_i^\perp
=
\#\{b\in D:\operatorname{wt}(b)=i\}.
\]

\begin{Proposition}\label{prop:crosspolytope-magic}
With the notation above,
\[
L_{R^\vee}(k)
=
\sum_{i=0}^d
A_i^\perp k^i(k+1)^{d-i}.
\]
In particular, $L_{R^\vee}(k)$ is magic positive.
\end{Proposition}

\begin{proof}
Since $\ZZ^d\subseteq\Lambda$, we have
\[
\Lambda^\vee\subseteq\ZZ^d.
\]
We first describe which elements of $\ZZ^d$ belong to $\Lambda^\vee$.

Let $z\in\ZZ^d$.  For each $x\in\Lambda$, let
$c\in C\subseteq\mathbb F_2^d$ be the image of $x$ under the map
\[
\pi:\Lambda\longrightarrow\mathbb F_2^d,
\qquad
\pi(x)=(2x_1,\ldots,2x_d)\pmod 2,
\]
and identify $c$ with its representative in $\{0,1\}^d$.
Then
\[
x=n+\frac12c
\]
for some $n\in\ZZ^d$.  Hence
\[
\langle z,x\rangle
=
z\cdot n+\frac12 z\cdot c.
\]
Since $z\cdot n\in\ZZ$, the condition
$\langle z,x\rangle\in\ZZ$ for every $x\in\Lambda$ is equivalent to
\[
z\cdot c\equiv0\pmod 2
\qquad
\text{for every }c\in C.
\]
Equivalently,
\[
(z\bmod 2)\cdot c=0
\qquad
\text{in }\mathbb F_2
\]
for every $c\in C$.  Thus
\[
\Lambda^\vee
=
\{z\in\ZZ^d:z\bmod2\in C^\perp\}
=
\{z\in\ZZ^d:z\bmod2\in D\}.
\]

Now let $k$ be a nonnegative even integer.  Since
\[
R^\vee=[-1,1]^d,
\]
we have
\[
kR^\vee\cap\Lambda^\vee
=
\{z\in[-k,k]^d\cap\ZZ^d:z\bmod2\in D\}.
\]
The interval $[-k,k]\cap\ZZ$ contains $k+1$ even integers and $k$
odd integers.  Therefore, for a fixed $b\in D$ with
$\operatorname{wt}(b)=i$, the number of $z\in[-k,k]^d\cap\ZZ^d$
satisfying
\[
z\bmod2=b
\]
is
\[
k^i(k+1)^{d-i}.
\]
Summing over all elements of $D$, we obtain
\[
L_{R^\vee}(k)
=
\sum_{i=0}^d
A_i^\perp k^i(k+1)^{d-i}
\]
for every nonnegative even integer $k$.

Both sides are polynomials in $k$, and they agree for infinitely many
values of $k$.  Hence the identity holds for all $k$.  Since
$A_i^\perp\ge0$ for every $i$, these are precisely the magic
coefficients of $L_{R^\vee}(k)$, and therefore
$L_{R^\vee}(k)$ is magic positive.
\end{proof}

\subsection{The slab polytopes \texorpdfstring{$Q_d$}{Q d}}\label{sec:slab}

For $d\ge2$, put
\[
Q_d
=
\left\{
x\in[-1,1]^d:
-1\le x_1+\cdots+x_d\le1
\right\}.
\]
The vertex description of $V_{2m}$ immediately gives
\[
Q_{2m}=V_{2m}^\vee.
\]
Konoike showed that $Q_{n-1}$ is lattice equivalent to
$\PP_{C_n}^\vee$ and obtained the following Ehrhart formula
\cite[Proposition~4.6 and its proof]{KonoikeCycles}:
\begin{equation}\label{eq:Q-ehrhart}
L_{Q_d}(k)
=
\sum_{j=0}^{\lfloor d/2\rfloor}
(-1)^j\binom{d+1}{j}
\binom{(d+1-2j)k+d-j}{d}.
\end{equation}
For even $d$, the same formula is also contained in Haase's explicit
Ehrhart formulas
\cite[Proposition~IV.2.7]{HaaseThesis}.

We now rewrite \eqref{eq:Q-ehrhart} in a form adapted to magic positivity.
Let
\[
d=2m+\varepsilon,
\qquad
\varepsilon\in\{0,1\},
\qquad
\alpha=\frac{1+\varepsilon}{2},
\]
and set
\[
\lambda_q=q+\alpha
\qquad(0\le q\le m).
\]

\begin{Proposition}\label{prop:Q-factorization}
There is a polynomial $C_{m,\varepsilon}(u)$ of degree $m$ such that
\[
L_{Q_{2m+\varepsilon}}(k)
=
(2k+1)^\varepsilon
C_{m,\varepsilon}(k(k+1)).
\]
More precisely,
\begin{equation}\label{eq:C-product}
\begin{aligned}
C_{m,\varepsilon}(u)
&=
\frac1{(2m+\varepsilon)!}
\sum_{q=0}^{m}
(-1)^{m-q}
\binom{2m+\varepsilon+1}{m-q}\lambda_q^\varepsilon
\prod_{a=0}^{m-1}
\left(
4\lambda_q^2u+\lambda_q^2-\lambda_a^2
\right).
\end{aligned}
\end{equation}
\end{Proposition}

\begin{proof}
In \eqref{eq:Q-ehrhart}, set $j=m-q$.  Since
\[
2m+\varepsilon+1-2j=2\lambda_q,
\qquad
2m+\varepsilon-j=m+q+\varepsilon,
\]
we get
\[
L_{Q_{2m+\varepsilon}}(k)
=
\sum_{q=0}^{m}
(-1)^{m-q}
\binom{2m+\varepsilon+1}{m-q}
\binom{2\lambda_q k+m+q+\varepsilon}{2m+\varepsilon}.
\]
Hence it is enough to show that 
\begin{equation}\label{eq:ETS}
\binom{2\lambda_q k+m+q+\varepsilon}{2m+\varepsilon}
=
\frac{\lambda_q^\varepsilon(2k+1)^\varepsilon}
{(2m+\varepsilon)!}
\prod_{a=0}^{m-1}
\left(
4\lambda_q^2u+\lambda_q^2-\lambda_a^2
\right),
\end{equation}
where $u=k(k+1)$.
Let $X=\lambda_q(2k+1)$.
Then 
\[
2\lambda_q k+m+q+\varepsilon
=
X+m+\frac{\varepsilon-1}{2}.
\]
If $\varepsilon=0$, then the left-hand side of \eqref{eq:ETS} can be written as
\[
\binom{X +m-\frac1 2}{2m}
=
\frac{(X +m-\frac1 2) \cdots (X -m+\frac1 2) }{(2m)!}
=
\frac{\prod_{a=0}^{m-1} (X + \lambda_a)(X - \lambda_a) }{(2m)!}.
\]
If $\varepsilon=1$,  then the left-hand side of \eqref{eq:ETS} can be written as
\[
\binom{X +m}{2m+1}
=
\frac{(X +m) \cdots X \cdots (X -m) }{(2m+1)!}
=
\frac{X \prod_{a=0}^{m-1} (X + \lambda_a)(X - \lambda_a) }{(2m+1)!}.
\]
Since
\[
X^2-\lambda_a^2
=
\lambda_q^2 (4k^2+4k +1)-\lambda_a^2
=
4\lambda_q^2u+\lambda_q^2-\lambda_a^2,
\]
substitution gives \eqref{eq:ETS} in both
cases.
\end{proof}

\section{Divided differences}\label{sec:dd}

We collect the standard facts about divided differences that are used later.
We refer to \cite{deBoorDividedDifferences,DubeauHermite} for general
background.  In particular, the mean-value theorem, the node-removal
identity, the divided differences of monomials, and the reciprocal formula
used below appear as Theorems~6, 7, 12, and~13 of \cite{DubeauHermite}.

Since we also use $[x^n]F(x)$ for coefficient extraction, we denote divided
differences by the symbol $\DD$.  The points at which a divided difference is
taken are called its \emph{nodes}.  For a function $f$ defined at pairwise
distinct nodes $y_0,\ldots,y_r$, set
\begin{equation} \label{didi_def}
    \DD[y_0,\ldots,y_r]f
:=
\sum_{a=0}^r
\frac{f(y_a)}{\prod_{b\ne a}(y_a-y_b)}.
\end{equation}
This formula shows that divided differences are symmetric in their nodes and
linear in the function $f$.  They also satisfy the recursive formula
\begin{equation} \label{didi_rec}
\DD[y_0,\ldots,y_r]f
=
\frac{
\DD[y_1,\ldots,y_r]f-
\DD[y_0,\ldots,y_{r-1}]f
}{y_r-y_0}.
\end{equation}
This is also a consequence of \cite[Theorem~7]{DubeauHermite}.

By linearity and \cite[Theorem~12(i)]{DubeauHermite}, if $P(t)$ is a
polynomial of degree at most $r-1$, then
\begin{equation}\label{didi_deg}
\DD[y_0,\ldots,y_r]P=0.
\end{equation}

\begin{Lemma}\label{lem:power-dd-positive}
Let $r\ge1$, and let $y_0,\ldots,y_r$ be pairwise distinct positive real
numbers.  Then
\[
\DD[y_0,\ldots,y_r]t^{r-1/2}>0.
\]
\end{Lemma}

\begin{proof}
By the mean-value theorem for divided differences
\cite[Theorem~6]{DubeauHermite}, there exists a point $\xi$ in the positive
interval spanned by the nodes such that
\[
\DD[y_0,\ldots,y_r]t^{r-1/2}
=
\frac1{r!}
\left.\frac{d^r}{dt^r}t^{r-1/2}\right|_{t=\xi}.
\]
For $t>0$,
\[
\frac{d^r}{dt^r}t^{r-1/2}
=
\left(r-\frac12\right)
\left(r-\frac32\right)\cdots\frac12\,t^{-1/2}>0.
\]
Thus we have $\DD[y_0,\ldots,y_r]t^{r-1/2}>0$.
\end{proof}

The following node-removal identity is obtained by repeated application of
\cite[Theorem~7]{DubeauHermite}.

\begin{Lemma}\label{lem:node-removal}
Let $y_0,y_1,\ldots,y_s$ be pairwise distinct, let
$L\subseteq\{1,\ldots,s\}$, and write
\[
\{1,\ldots,s\}\setminus L=\{\ell_1<\cdots<\ell_q\}.
\]
For every function $h$ defined at $y_0,y_1,\ldots,y_s$,
\[
\DD[y_0,y_1,\ldots,y_s]
\left(
\prod_{b\in L}(t-y_b)h(t)
\right)
=
\DD[y_0,y_{\ell_1},\ldots,y_{\ell_q}]h.
\]
\end{Lemma}

\begin{proof}
Using \eqref{didi_def}, the summands corresponding to nodes $y_b$
with $b\in L$ vanish.
In each remaining summand, the factors corresponding to the removed
nodes cancel between the numerator and denominator.
The resulting sum is precisely the divided difference on the
right-hand side.
\end{proof}

For example,
\[
\DD[y_0,y_1,y_2]\bigl((t-y_1)h(t)\bigr)
=
\DD[y_0,y_2]h.
\]
We also record the divided difference of a reciprocal linear factor.

\begin{Lemma}\label{lem:reciprocal-dd}
Let $y_0,\ldots,y_r$ be pairwise distinct and suppose that
$y_a+u\ne0$ for $0\le a\le r$.  Then
\begin{equation}\label{t+u}
\DD[y_0,\ldots,y_r]\frac1{t+u}
=
\frac{(-1)^r}{\prod_{a=0}^r(y_a+u)}.
\end{equation}
\end{Lemma}

\begin{proof}
Set $x_a=y_a+u$.  Formula \eqref{didi_def} immediately gives translation invariance in
the form
\[
\DD[y_0,\ldots,y_r]f(t+u)
=
\DD[x_0,\ldots,x_r]f(t).
\]
Applying \cite[Theorem~13]{DubeauHermite} to $f(t)=1/t$, we obtain
\[
\DD[y_0,\ldots,y_r]\frac1{t+u}
=
\DD[x_0,\ldots,x_r]\frac1t
=
\frac{(-1)^r}{x_0x_1\cdots x_r}
=
\frac{(-1)^r}{\prod_{a=0}^r(y_a+u)}.
\]
\end{proof}

The following form will be used in the proof of Lemma~\ref{lem:dd-positive}.

\begin{Corollary}\label{cor:rational-dd}
Let $y_0,\ldots,y_r$ be pairwise distinct and suppose that $y_a+u\ne0$
for all $a$.  Then
\[
\DD[y_0,\ldots,y_r]
\frac{\prod_{b=1}^r(t+y_b)}{t+u}
=
\frac{\prod_{b=1}^r(u-y_b)}{\prod_{a=0}^r(u+y_a)}.
\]
\end{Corollary}

\begin{proof}
For $r=0$, the identity is simply evaluation at $y_0$.  Assume $r\ge1$ and set
\[
q(t):=\prod_{b=1}^r(t+y_b).
\]
Since $q(t)-q(-u)$ is divisible by $t+u$, there is a polynomial $Q_u(t)$
of degree at most $r-1$ such that
\[
q(t)=(t+u)Q_u(t)+q(-u).
\]
Hence
\[
\frac{q(t)}{t+u}
=
Q_u(t)+\frac{q(-u)}{t+u}.
\]
By \eqref{didi_deg}, the $r$-th divided difference of $Q_u$ vanishes.  Therefore
Lemma~\ref{lem:reciprocal-dd} gives
\[
\begin{aligned}
\DD[y_0,\ldots,y_r]\frac{q(t)}{t+u}
&=q(-u)\DD[y_0,\ldots,y_r]\frac1{t+u}\\
&=\left(\prod_{b=1}^r(y_b-u)\right)
\frac{(-1)^r}{\prod_{a=0}^r(y_a+u)}\\
&=\frac{\prod_{b=1}^r(u-y_b)}{\prod_{a=0}^r(u+y_a)}.
\end{aligned}
\]
\end{proof}

\section{Local decompositions}\label{sec:local-expansion}

We use the rising factorial
\[
(a)_j=a(a+1)\cdots(a+j-1),
\qquad
(a)_0=1.
\]

\subsection{A local expansion}

Let $S_m:=\{0,1,\ldots,m-1\}$.
Continue to use the notation
\[
\alpha=\frac{1+\varepsilon}{2},
\qquad
\lambda_q=q+\alpha.
\]
For $J\subseteq S_m$, write $r=|J|$ and define
\begin{equation}\label{eq:parity-weight}
\omega_{m,\varepsilon}(J)
:=
4^r(m+\alpha)
\DD[\lambda_m^2,(\lambda_j^2)_{j\in J}]
t^{r-1/2}.
\end{equation}

\begin{Proposition}\label{prop:unified-local-expansion}
For $\varepsilon\in\{0,1\}$,
\[
C_{m,\varepsilon}(u)
=
\sum_{J\subseteq S_m}
\omega_{m,\varepsilon}(J)u^{|J|}.
\]
Moreover, $\omega_{m,\varepsilon}(\emptyset)=1$
and
\[
\omega_{m,\varepsilon}(J)>0
\qquad
(J\subseteq S_m).
\]
\end{Proposition}

\begin{proof}
Since 
\[
\begin{aligned}
\prod_{a=0}^{m-1}
\left(
4\lambda_q^2u+\lambda_q^2-\lambda_a^2
\right)
=
\sum_{J \subseteq S_m} 
\left(
4^{|J|} \lambda_q^{2|J|} u^{|J|}
\prod_{a\in S_m\setminus J}
(\lambda_q^2-\lambda_a^2)
\right),
\end{aligned}
\]
we have
\[
C_{m,\varepsilon}(u)
=
\sum_{J \subseteq S_m} \nu_{m,\varepsilon}(J) 
u^{|J|},
\]
where
\[
\nu_{m,\varepsilon}(J)
=
\frac{4^{|J|}}{(2m+\varepsilon)!}
\sum_{q=0}^{m}
(-1)^{m-q}
\binom{2m+\varepsilon+1}{m-q}
\lambda_q^{2|J|+\varepsilon}
\prod_{a\in S_m\setminus J}
(\lambda_q^2-\lambda_a^2).
\]
If $q\in S_m\setminus J$, the product contains
$\lambda_q^2-\lambda_q^2=0$.  
Thus we have
\[
\nu_{m,\varepsilon}(J)
=
\frac{4^{|J|}}{(2m+\varepsilon)!}
\sum_{q \in J\cup\{m\}}
(-1)^{m-q}
\binom{2m+\varepsilon+1}{m-q}
\lambda_q^{2|J|+\varepsilon}
\prod_{a\in S_m\setminus J}
(\lambda_q^2-\lambda_a^2).
\]
For $0\le q\le m$,
\begin{equation}\label{eq:lambda-product}
\begin{aligned}   
\prod_{\substack{0\le a\le m\\a\ne q}}
(\lambda_q^2-\lambda_a^2)
&=
\prod_{\substack{0\le a\le m\\a\ne q}}
(q-a)(q+a+2\alpha)\\
&=
(-1)^{m-q}q!(m-q)!
\frac{(q+2\alpha)_{m+1}}{2\lambda_q}
\end{aligned}
\end{equation}
since
\[
\begin{aligned}
\prod_{\substack{0\le a\le m\\a\ne q}}(q-a)
&=
q (q-1) \cdots 1 \cdot (-1) \cdots (-m+q) 
=
(-1)^{m-q}q!(m-q)!,\\
\prod_{\substack{0\le a\le m\\a\ne q}}(q+a+2\alpha)
&=
\frac{(q+2\alpha)_{m+1}}{2\lambda_q}.
\end{aligned}
\]
For $q\in J \cup \{m\}$,
\[
\prod_{a\in S_m\setminus J}
(\lambda_q^2-\lambda_a^2)
=
\frac{
\displaystyle
\prod_{\substack{0\le a\le m\\a\ne q}}
(\lambda_q^2-\lambda_a^2)
}{
\displaystyle
\prod_{\substack{c\in  J \cup \{m\}\\c\ne q}}
(\lambda_q^2-\lambda_c^2)
}.
\]
Using \eqref{eq:lambda-product}, together with
\[
\begin{aligned}
(q+2\alpha)_{m+1}
&=
(q+\varepsilon +1)_{m+1}=
\frac{(m+q+\varepsilon+1)!}{(q+\varepsilon)!},\\
\frac{q!}{(q+\varepsilon)!}
&=
\left\{
\begin{array}{cccc}
1 &=&\lambda_q^{-\varepsilon} & \mbox{if } \varepsilon=0,\\
\\
\frac{1}{q+1}&=&\lambda_q^{-\varepsilon}& \mbox{if } \varepsilon=1,
\end{array}
\right.
\\
\frac{(2m+\varepsilon+1)!}{(2m+\varepsilon)!}
&=
2m+\varepsilon+1=
2(m+\alpha),
\end{aligned}
\]
we obtain
\[
\nu_{m,\varepsilon}(J)
=
4^{|J|}(m+\alpha)
\sum_{q\in  J \cup \{m\}}
\frac{\lambda_q^{2|J|-1}}
{\displaystyle
\prod_{\substack{c\in  J \cup \{m\}\\c\ne q}}
(\lambda_q^2-\lambda_c^2)}.
\]
The last sum is precisely the divided difference
\[
\DD[\lambda_m^2,(\lambda_j^2)_{j\in J}]
t^{|J|-1/2}.
\]
Thus
\[
\nu_{m,\varepsilon}(J)=\omega_{m,\varepsilon}(J),
\]
which proves the expansion.

For $J=\emptyset$,
\[
\omega_{m,\varepsilon}(\emptyset)
=
(m+\alpha)\DD[(m+\alpha)^2]t^{-1/2}=1.
\]
For $J\ne\emptyset$, positivity follows from
Lemma~\ref{lem:power-dd-positive}.  Hence every weight is positive.
\end{proof}

\subsection{Magic coefficients and their local decomposition}\label{sec:magic-local}

The positivity in Proposition~\ref{prop:unified-local-expansion}
does not by itself imply magic positivity, since the change to the
magic basis introduces alternating signs.
We therefore decompose the magic coefficients into local contributions,
working with arbitrary distinct positive nodes with one distinguished
maximum.

Let $m\ge1$, let $S_m=\{0,\ldots,m-1\}$, and let
\[
X,x_0,\ldots,x_{m-1}
\]
be pairwise distinct positive real numbers such that
\[
X>x_j
\qquad(j\in S_m).
\]
Write $\mathbf x=(X,x_0,\ldots,x_{m-1})$.
For $J\subseteq S_m$, define
\begin{equation}\label{eq:generic-weight}
\omega_{\mathbf x}(J)
:=
4^{|J|}X
\DD[X^2,(x_j^2)_{j\in J}]t^{|J|-1/2},
\qquad
\Omega_{\mathbf x}(u)
:=
\sum_{J\subseteq S_m}\omega_{\mathbf x}(J)u^{|J|}.
\end{equation}
The magic transform of the polynomial
$\Omega_{\mathbf x}(k(k+1))$ is
\[
\mathcal M_{\mathbf x}(y)
:=
(1-y)^{2m}
\Omega_{\mathbf x}\left(\frac{y}{(1-y)^2}\right).
\]
By \eqref{eq:generic-weight},
\begin{equation}\label{eq:generic-M}
\mathcal M_{\mathbf x}(y)
=
\sum_{J\subseteq S_m}
\omega_{\mathbf x}(J)y^{|J|}(1-y)^{2m-2|J|}.
\end{equation}
Write
\[
\mathcal M_{\mathbf x}(y)
=
\sum_{s=0}^{2m}d_s y^s.
\]
For $I\subseteq S_m$ with $|I|=s\le m$, define
\begin{equation}\label{eq:generic-tau}
\tau_{\mathbf x}(I)
:=
\sum_{J\subseteq I}
(-1)^{s-|J|}
\frac{\binom{2m-2|J|}{s-|J|}}
{\binom{m-|J|}{s-|J|}}
\omega_{\mathbf x}(J).
\end{equation}

\begin{Proposition}\label{prop:generic-localization}
For $0\le s\le m$,
\[
d_s
=
\sum_{\substack{I\subseteq S_m\\|I|=s}}
\tau_{\mathbf x}(I).
\]
Moreover, $d_s=d_{2m-s}$.
\end{Proposition}

\begin{proof}
The coefficient of $y^s$ in the summand of \eqref{eq:generic-M} indexed by
$J$ is
\[
(-1)^{s-|J|}
\binom{2m-2|J|}{s-|J|}
\omega_{\mathbf x}(J).
\]
A fixed $J$ with $|J|\le s$ is contained in exactly
\[
\binom{m-|J|}{s-|J|}
\]
subsets $I\subseteq S_m$ of cardinality $s$.  Summing
\eqref{eq:generic-tau} over all such $I$ therefore gives $d_s$.

For the symmetry, each summand in \eqref{eq:generic-M} satisfies
\[
y^{2m}
\left(
(1/y)^{|J|}(1-1/y)^{2m-2|J|}
\right)
=
y^{|J|}(1-y)^{2m-2|J|},
\]
since $2m-2|J|$ is even.  
Hence $y^{2m}\mathcal M_{\mathbf x}(1/y)=\mathcal M_{\mathbf x}(y)$.
\end{proof}

\subsection{A divided-difference formula for the local coefficients}\label{sec:local-dd}

For $0\le q\le s$, let
\[
e_q(z_1,\ldots,z_s)
=
\sum_{\substack{A\subseteq\{1,\ldots,s\}\\|A|=q}}
\prod_{b\in A}z_b,
\qquad
e_0=1,
\]
be the elementary symmetric polynomial.  
For the terminating hypergeometric identity below, see
\cite[Chapter~2]{AndrewsAskeyRoy}.
For $0\le j\le s\le m$, define
\[
\beta_{m,s,j}
:=
\sum_{\ell=0}^j
(-1)^{j-\ell}\binom{j}{\ell}4^\ell
\frac{\binom{2m-2\ell}{s-\ell}}
{\binom{m-\ell}{s-\ell}}.
\]

\begin{Lemma}\label{lem:beta-identity}
Let $a=m-s+\frac12$.
Then
\[
\beta_{m,s,j}
=
\frac{\binom{2m}{s}}{\binom{m}{s}}
\frac{(a)_j(a)_{s-j}}{(a)_s}.
\]
In particular,
\[
\beta_{m,s,j}>0
\qquad\text{and}\qquad
\beta_{m,s,j}=\beta_{m,s,s-j}.
\]
\end{Lemma}

\begin{proof}
Set
\[
C_\ell
=
\frac{\binom{2m-2\ell}{s-\ell}}
{\binom{m-\ell}{s-\ell}},
\qquad
C_0=\frac{\binom{2m}{s}}{\binom{m}{s}}.
\]
A direct comparison of successive ratios gives
\[
\frac{4^\ell C_\ell}{C_0}
=
\frac{(1-2a-s)_\ell}{(1-a-s)_\ell}.
\]
Therefore
\[
\frac{\beta_{m,s,j}}{C_0}
=
(-1)^j
\sum_{\ell=0}^j
\frac{(-j)_\ell(1-2a-s)_\ell}
{(1-a-s)_\ell\,\ell!}
=
(-1)^j\frac{(a)_j}{(1-a-s)_j},
\]
where the second equality is the terminating Chu--Vandermonde identity.
Finally,
\[
(1-a-s)_j
=
(-1)^j(a+s-j)_j
=
(-1)^j\frac{(a)_s}{(a)_{s-j}},
\]
which proves the formula.
\end{proof}

\begin{Proposition}\label{prop:generic-local-dd}
Let $I=\{i_1<\cdots<i_s\}\subseteq S_m$,
and put
\[
z_0=X^2,
\qquad
z_b=x_{i_b}^2
\quad(1\le b\le s).
\]
Then
\begin{equation}\label{eq:generic-local-dd}
\tau_{\mathbf x}(I)
=
X C_{m,s}
\DD[z_0,\ldots,z_s]
\frac{
\displaystyle
\sum_{j=0}^s
(a)_j(a)_{s-j}
e_{s-j}(z_1,\ldots,z_s)t^j
}{\sqrt t},
\end{equation}
where
\[
a=m-s+\frac12,
\qquad
C_{m,s}
=
\frac{\binom{2m}{s}}
{\binom{m}{s}(a)_s}>0.
\]
\end{Proposition}

\begin{proof}
For $J\subseteq I$, let
\[
\sigma(J)
=
\{b\in\{1,\ldots,s\}:i_b\in J\}.
\]
If $|J|=r$, then by Lemma~\ref{lem:node-removal},
\[
\begin{aligned}
\omega_{\mathbf x}(J)
&=
4^rX
\DD[z_0,(z_b)_{b\in\sigma(J)}]t^{r-1/2}\\
&=
4^rX
\DD[z_0,z_1,\ldots,z_s]
\left(
t^{r-1/2}
\prod_{b\notin\sigma(J)}(t-z_b)
\right).
\end{aligned}
\]
Substituting this into \eqref{eq:generic-tau}, grouping by $r=|J|$, and
using
\[
\sum_{\substack{J\subseteq I\\|J|=r}}
\prod_{b\notin\sigma(J)}(t-z_b)
=
e_{s-r}(t-z_1,\ldots,t-z_s),
\]
we obtain
\[
\tau_{\mathbf x}(I)
=
X\DD[z_0,\ldots,z_s]
\frac{\widetilde P(t)}{\sqrt t},
\]
where
\[
\widetilde P(t)
=
\sum_{r=0}^s
(-1)^{s-r}
\frac{\binom{2m-2r}{s-r}}
{\binom{m-r}{s-r}}
(4t)^r
e_{s-r}(t-z_1,\ldots,t-z_s).
\]
The elementary identity
\[
e_{s-r}(t-z_1,\ldots,t-z_s)
=
\sum_{j=r}^s
(-1)^{s-j}\binom{j}{r}
e_{s-j}(z_1,\ldots,z_s)t^{j-r}
\]
shows that the coefficient of
\[
e_{s-j}(z_1,\ldots,z_s)t^j
\]
in $\widetilde P(t)$ is $\beta_{m,s,j}$.  The result now follows from
Lemma~\ref{lem:beta-identity}.
\end{proof}

\section{Positivity of the local coefficients}\label{sec:plus}
\subsection{A positivity lemma for divided differences}

The following inequality is the key positivity statement used below.

\begin{Lemma}\label{lem:dd-positive}
Let $r\ge0$ and $0<y_r<\cdots<y_1<y_0$.
Then
\[
\DD[y_0,\ldots,y_r]
\frac{\prod_{b=1}^r(t+y_b)}{\sqrt t}
>0.
\]
\end{Lemma}

\begin{proof}
The case $r=0$ is immediate. Assume $r\ge1$ and set
\[
q(t):=\prod_{b=1}^r(t+y_b),
\qquad
K_r:=
\DD[y_0,\ldots,y_r]
\frac{q(t)}{\sqrt t}.
\]
We use the elementary formula
\[
\frac1{\sqrt t}
=
\frac2\pi
\int_0^\infty
\frac{dv}{t+v^2},
\qquad t>0,
\]
which follows from Euler's beta integral;
see, for example, \cite[Chapter~1]{AndrewsAskeyRoy}.
Since the integral converges absolutely at every positive node and a
divided difference at fixed distinct nodes is a finite linear combination
of node evaluations, we may interchange the finite sum and the integral.
Hence
\[
K_r
=
\frac2\pi
\int_0^\infty
\DD[y_0,\ldots,y_r]
\frac{q(t)}{t+v^2}
\,dv.
\]
By Corollary~\ref{cor:rational-dd},
\[
\DD[y_0,\ldots,y_r]
\frac{q(t)}{t+v^2}
=
\frac{\prod_{b=1}^r(v^2-y_b)}{\prod_{a=0}^r(v^2+y_a)}.
\]
Writing
\[
y_a=x_a^2,
\qquad
x_0>x_1>\cdots>x_r>0,
\]
we obtain
\[
K_r
=
\frac2\pi
\int_0^\infty
f(v)\,dv,
\]
where
\[
f(v)
:=
\frac1{v^2+x_0^2}
\prod_{b=1}^r
\frac{v^2-x_b^2}
     {v^2+x_b^2}.
\]
By substituting $v= x_0^2 / w$,
we have
\[
\int_0^{x_0}
f(v)\,dv
=
\int_{x_0}^{\infty}
f\!\left(\frac{x_0^2}{w}\right)
     \, \frac{x_0^2}{w^2} dw.
\]
Renaming the integration variable \(w\) as \(v\), we obtain
\begin{equation}\label{8.2}
\begin{aligned}
\int_0^\infty f(v)\,dv
&=
\int_{x_0}^\infty
\left(
f(v)
+
\frac{x_0^2}{v^2}
f\!\left(\frac{x_0^2}{v}\right)
\right)\,dv\\
&=
\int_{x_0}^\infty
\frac1{v^2+x_0^2}
\left(
\prod_{b=1}^r
\frac{v^2-x_b^2}
     {v^2+x_b^2}
+
\prod_{b=1}^r
\frac{x_0^4-x_b^2v^2}
     {x_0^4+x_b^2v^2}
\right)\,dv\\
&=
\int_{x_0}^\infty
\frac1{v^2+x_0^2}
\left(
\prod_{b=1}^r A_b
+
\prod_{b=1}^r B_b
\right)\,dv,
\end{aligned}
\end{equation}
where
\[
T:=\frac{v^2}{x_0^2}\ge1,
\qquad
s_b:=\frac{x_b^2}{x_0^2}\in(0,1),
\qquad
A_b:=\frac{T-s_b}{T+s_b},
\qquad
B_b:=\frac{1-s_bT}{1+s_bT}
\]
for $v\ge x_0$ and $b = 1,\dots,r$.
Since \(T\ge1\) and \(0<s_b<1\), we have
\[
2s_b(T^2-1)\ge0, \qquad
2T(1-s_b^2)\ge0.
\]
Thus it follows that
\[
B_b\le A_b,
\qquad
-A_b\le B_b,
\]
and hence
\[
|B_b|\le A_b.
\]
Therefore
\(
\left|
\prod_{b=1}^r B_b
\right|
\le
\prod_{b=1}^r A_b,
\)
and so
\[
\prod_{b=1}^r A_b
+
\prod_{b=1}^r B_b
\ge0.
\]
Thus the integrand in \eqref{8.2} is nonnegative for all $v\ge x_0$.
At $v=x_0$, equivalently $T=1$, we have
\[
A_b=B_b=\frac{1-s_b}{1+s_b}>0,
\]
and hence the integrand in \eqref{8.2} is strictly positive there. By continuity it is positive
on a neighborhood of $x_0$. Hence the integral in \eqref{8.2} is
strictly positive.
Consequently,
\[
K_r
=
\frac2\pi
\int_0^\infty f(v)\,dv
>0,
\]
as desired.
\end{proof}

\subsection{Finite symmetrization and local positivity}\label{sec:local-positive}

To apply Lemma~\ref{lem:dd-positive} to the expression in
Proposition~\ref{prop:generic-local-dd}, we rewrite its numerator as
a positive linear combination of products involving factors
$t-z_b$ and $t+z_b$.
The factors $t-z_b$ can then be eliminated by
Lemma~\ref{lem:node-removal}.
We obtain the required identity using binary Krawtchouk polynomials. For a fixed $s\ge0$
and $0\le q,j\le s$, define
\begin{equation}\label{9.1}
\kappa_q^{(s)}(j)
:=[z^q](1-z)^j(1+z)^{s-j}.
\end{equation}
Equivalently,
\begin{equation}\label{9.2}
\kappa_q^{(s)}(j)
=
\sum_{h=0}^q
(-1)^h\binom{j}{h}\binom{s-j}{q-h}.
\end{equation}
These are the binary Krawtchouk polynomials in the normalization used by
Feinsilver and Kocik; see \cite{FeinsilverKocik}.  Equivalently, they are
the entries of the Krawtchouk matrix of order $s$. 

The following is the standard orthogonality relation for these
polynomials, written in the dual form needed below.  We include the short
derivation for completeness.  Here $\delta_{i,j}$ denotes the Kronecker
delta.

\begin{Lemma}\label{lem:kappa-orthogonality} 
For $0\le i,j\le s$, 
\begin{equation}\label{9.3}
\frac{1}{2^s}
\sum_{q=0}^s 
\frac{\kappa_q^{(s)}(i)\kappa_q^{(s)}(j)} {\binom{s}{q}} = \frac{\delta_{i,j}}{\binom{s}{i}}. 
\end{equation} 
\end{Lemma}

\begin{proof}
By the binomial theorem,
\[
\begin{aligned}
 &\quad \sum_{j=0}^s\binom{s}{j}
(1-z)^j(1+z)^{s-j}
(1-w)^j(1+w)^{s-j}\\
& =
\left((1-z)(1-w)+(1+z)(1+w)\right)^s\\
& =
2^s(1+zw)^s.
\end{aligned}
\]
Comparing coefficients of $z^pw^q$ and using \eqref{9.1}, we obtain
\[ 
\sum_{j=0}^s
\binom{s}{j}
\kappa_p^{(s)}(j)\kappa_q^{(s)}(j)
=
2^s\binom{s}{q}\delta_{p,q}.
\]
Let
\[
K=(\kappa_q^{(s)}(j))_{0\le j,q\le s},
\qquad
B=\operatorname{diag}\!\left(
\binom{s}{0},\ldots,\binom{s}{s}
\right).
\]
The preceding relation is
\[
K^{\mathsf T}BK=2^sB.
\]
In particular, $K$ is invertible.  Taking inverses gives
\[
K^{-1}B^{-1}(K^{\mathsf T})^{-1}=2^{-s}B^{-1}.
\]
Multiplying on the left by $K$ and on the right by $K^{\mathsf T}$, we
obtain
\[
KB^{-1}K^{\mathsf T}=2^sB^{-1},
\]
which is precisely \eqref{9.3}.
\end{proof}

We now compute the Krawtchouk expansion of the specific symmetric weights
that occur in the local positivity argument.  This is the specialization
needed for the finite symmetrization identity.

\begin{Lemma} \label{lem:finite-symmetrization}
Let $a>0$ and $s\ge0$.  For $0\le j\le s$, put
\[
W_j:=\frac{(a)_j(a)_{s-j}}{(2a)_s}.
\]
Then
\begin{equation}\label{9.4}
W_j
=
\frac1{2^s}
\sum_{\ell=0}^{\lfloor s/2\rfloor}
\frac{(1/2)_\ell}{(a+1/2)_\ell}
\kappa_{2\ell}^{(s)}(j).
\end{equation}
Consequently, for indeterminates $t,y_1,\ldots,y_s$,
\begin{equation}\label{9.5}
\begin{aligned}
&\sum_{j=0}^s
(a)_j(a)_{s-j}e_{s-j}(y_1,\ldots,y_s)t^j\\
&\quad=
\frac{(2a)_s}{2^s}
\sum_{\substack{L\subseteq\{1,\ldots,s\}\\|L|\text{ even}}}
\frac{(1/2)_{|L|/2}}{(a+1/2)_{|L|/2}}
\prod_{b\in L}(t-y_b)
\prod_{b\notin L}(t+y_b).
\end{aligned}
\end{equation}
\end{Lemma}

\begin{proof}
We first prove \eqref{9.4}.  Consider the (unnormalized) Krawtchouk
transform
\[
\widehat W_q
:=
\sum_{j=0}^s\binom{s}{j}W_j\kappa_q^{(s)}(j).
\]
By \eqref{9.1},
\begin{equation}\label{9.6}
\sum_{q=0}^s\widehat W_qz^q
=
\sum_{j=0}^s
\binom{s}{j}
\frac{(a)_j(a)_{s-j}}{(2a)_s}
(1-z)^j(1+z)^{s-j}.
\end{equation}
Using the formal power series
\[
(1-XT)^{-a}
=
\sum_{n\ge0}\frac{(a)_n}{n!}X^nT^n,
\]
we obtain, for arbitrary $X,Y$,
\begin{equation}\label{9.7}
\begin{aligned}
\sum_{j=0}^s
\binom{s}{j}(a)_j(a)_{s-j}X^jY^{s-j}
&=
s! 
\sum_{j=0}^s
\frac{(a)_j}{j!}X^j\frac{(a)_{s-j}}{(s-j)!}Y^{s-j}
\\
&=
s![T^s](1-XT)^{-a}(1-YT)^{-a}.
\end{aligned}
\end{equation}
Substituting $X=1-z$ and $Y=1+z$ into \eqref{9.7}, 
expression \eqref{9.6} becomes
\[
\frac{s!}{(2a)_s}[T^s]
\left((1-T)^2-z^2T^2\right)^{-a}.
\]
Applying  
\[
(1-X)^{-a}
=
\sum_{\ell \ge0}\frac{(a)_\ell}{\ell!}X^\ell,
\]
for $X = \frac{z^2T^2}{(1-T)^2}$, we have
\[
\begin{aligned}
\left((1-T)^2-z^2T^2\right)^{-a}
&=
(1-T)^{-2a}
\left(1-\frac{z^2T^2}{(1-T)^2}\right)^{-a}\\
&=
\sum_{\ell\ge0}
\frac{(a)_\ell}{\ell!}
z^{2\ell}T^{2\ell}(1-T)^{-2a-2\ell}.
\end{aligned}
\]
Since only even powers of \(z\) occur in the above expansion,
\[
\widehat W_{2\ell+1}=0.
\]
In addition, for $0\le2\ell\le s$, 
since
\[
[T^s] T^{2\ell}(1-T)^{-2a-2\ell}
=
[T^{s-2\ell }] \sum_{n \ge0}\frac{(2 a + 2\ell)_n}{n!}T^n
=
\frac{(2 a + 2\ell)_{s-2\ell}}{(s-2\ell)!}
\]
and since
\[
(2a)_s=(2a)_{2\ell}(2a+2\ell)_{s-2\ell},
\]
we have
\[
\widehat W_{2\ell}
=
\frac{s!}{(2a)_s}
\frac{(a)_\ell}{\ell!}
\frac{(2a+2\ell)_{s-2\ell}}{(s-2\ell)!}\\
=
\frac{s!}{(s-2\ell)!}
\frac{(a)_\ell}
     {\ell!\,(2a)_{2\ell}}
=
\binom{s}{2\ell}
\frac{(2\ell)!}{\ell!}
\frac{(a)_\ell}
     {(2a)_{2\ell}}.
\]
Using the duplication identities
\[
(2a)_{2\ell}
=
2^{2\ell}(a)_\ell
\left(a+\frac12\right)_\ell,
\qquad
(2\ell)!
=
2^{2\ell}\ell!
\left(\frac12\right)_\ell,
\]
it follows that
\[
\frac{(2\ell)!}{\ell!}
\frac{(a)_\ell}
     {(2a)_{2\ell}}
=
\frac{
2^{2\ell}(\frac12)_\ell (a)_\ell
}{
2^{2\ell}(a)_\ell (a+\frac12)_\ell
}
=
\frac{(\frac12)_\ell}
     {(a+\frac12)_\ell}.
\]
Hence
\begin{equation}\label{9.8}
\widehat W_{2\ell}
=
\binom{s}{2\ell}
\frac{(\frac12)_\ell}
     {(a+\frac12)_\ell}.
\end{equation}
Then
\[
\begin{aligned}   
\sum_{q=0}^s
\frac{\kappa_q^{(s)}(i)}{2^s\binom{s}{q}}
\widehat W_q
&=
\sum_{q=0}^s
\frac{\kappa_q^{(s)}(i)}{2^s\binom{s}{q}}
\sum_{j=0}^s\binom{s}{j}W_j\kappa_q^{(s)}(j)\\
&=
\sum_{j=0}^s
\binom{s}{j}W_j
\sum_{q=0}^s
\frac{\kappa_q^{(s)}(i)\kappa_q^{(s)}(j)}{2^s\binom{s}{q} }.
\end{aligned}
\]
By Lemma~\ref{lem:kappa-orthogonality},
this equals 
\[
\sum_{j=0}^s
\binom{s}{j}W_j
\frac{\delta_{i,j}}{\binom{s}{i}}=W_i .
\]
Substituting $\widehat W_{2\ell +1} =0$ and  \eqref{9.8},
we have
\[
W_i 
= 
\sum_{q=0}^s
\frac{\kappa_q^{(s)}(i)}{2^s\binom{s}{q}}
\widehat W_q
=
\sum_{\ell=0}^{\lfloor s/2\rfloor}
\frac{\kappa_{2\ell}^{(s)}(i)}{2^s\binom{s}{2\ell}}
\widehat W_{2\ell}
=
\sum_{\ell=0}^{\lfloor s/2\rfloor}
\frac{\kappa_{2\ell}^{(s)}(i)}{2^s}
\frac{(\frac12)_\ell}
     {(a+\frac12)_\ell}
\]
which yields \eqref{9.4}.

We next prove \eqref{9.5}. 
For $L\subseteq\{1,\ldots,s\}$,
set
\[
F_L(t,y)
:=
\prod_{b\in L}(t-y_b)
\prod_{b\notin L}(t+y_b).
\]
Fix \(j\) and \(\ell\), and consider
\[
\sum_{|L|=2\ell}F_L(t,y).
\]
This polynomial is symmetric in \(y_1,\ldots,y_s\), and it is multi-affine in these variables. Its coefficient of \(t^j\) is therefore a symmetric multi-affine homogeneous polynomial of degree \(s-j\). Every monomial in this coefficient is a product of \(s-j\) distinct variables, and symmetry forces all such monomials to have the same coefficient. Hence it is a scalar multiple of \(e_{s-j}(y_1,\ldots,y_s)\).

To obtain a monomial contributing to
\(e_{s-j}(y)t^j\), we must choose \(s-j\) of the
\(s\) factors to contribute a \(y\)-term and the
remaining \(j\) factors to contribute \(t\).
Suppose that among the chosen \(s-j\) factors,
exactly \(h\) belong to \(L\).
Each such choice contributes the sign \((-1)^h\).
The number of possibilities is
\[
\binom{s-j}{h}\binom{j}{2\ell-h},
\]
since \(h\) indices of \(L\) must lie among the
\(s-j\) chosen \(y\)-positions and the remaining
\(2\ell-h\) indices of \(L\) among the \(j\)
chosen \(t\)-positions.
Hence the coefficient of
\(e_{s-j}(y)t^j\) in
$\sum_{|L|=2\ell}F_L(t,y)$
is
\[
\sum_h
(-1)^h
\binom{s-j}{h}
\binom{j}{2\ell-h}.
\]
By the defining formula \eqref{9.2}, this equals
$\kappa_{2\ell}^{(s)}(s-j)$.
Since
\[
\kappa_q^{(s)}(s-j)
=
(-1)^q\kappa_q^{(s)}(j),
\]
and \(q=2\ell\) is even, we obtain
\[
\kappa_{2\ell}^{(s)}(s-j)
=
\kappa_{2\ell}^{(s)}(j).
\]
Therefore the coefficient of
\(e_{s-j}(y)t^j\) on the right-hand side of
\eqref{9.5} is
\[
\frac{(2a)_s}{2^s}
\sum_{\ell=0}^{\lfloor s/2\rfloor}
\frac{(1/2)_\ell}{(a+1/2)_\ell}
\kappa_{2\ell}^{(s)}(j).
\]
By \eqref{9.4}, this equals
\[
(2a)_sW_j
=
(a)_j(a)_{s-j}.
\]
Hence the coefficients of
\(e_{s-j}(y_1,\ldots,y_s)t^j\)
coincide on both sides of \eqref{9.5}, which proves
the identity.
\end{proof}

\begin{Theorem}\label{thm:node-positivity}
Let $m\ge1$ and let $X,x_0,\ldots,x_{m-1}$
be pairwise distinct positive real numbers satisfying $X>x_j$ for every
$j$.  Then $\tau_{\mathbf x}(I)>0$ for every $I\subseteq S_m$.
Consequently, every coefficient of
\[
\mathcal M_{\mathbf x}(y)
=
(1-y)^{2m}
\Omega_{\mathbf x}\left(\frac{y}{(1-y)^2}\right)
\]
is strictly positive.
\end{Theorem}

\begin{proof}
Fix $I=\{i_1<\cdots<i_s\}\subseteq S_m$ and use the notation of
Proposition~\ref{prop:generic-local-dd}.  In particular,
\[
a=m-s+\frac12>0.
\]
Applying Lemma~\ref{lem:finite-symmetrization} to the numerator in
\eqref{eq:generic-local-dd}, and then using
Lemma~\ref{lem:node-removal}, gives
\[
\begin{aligned}
\tau_{\mathbf x}(I)
&=
\frac{XC_{m,s}(2a)_s}{2^s}
\sum_{\substack{L\subseteq\{1,\ldots,s\}\\|L|\text{ even}}}
\frac{(1/2)_{|L|/2}}{(a+1/2)_{|L|/2}} 
\DD[z_0,(z_b)_{b\notin L}]
\frac{\prod_{b\notin L}(t+z_b)}{\sqrt t}.
\end{aligned}
\]
The distinguished node $z_0=X^2$ is larger than all remaining positive
nodes.  By symmetry of divided differences, the latter may be arranged in
decreasing order, and Lemma~\ref{lem:dd-positive} shows that every divided
difference in the sum is strictly positive.  All other factors are positive
as well.  
Hence $\tau_{\mathbf x}(I)>0$.

Proposition~\ref{prop:generic-localization} now gives $d_s>0$ for
$0\le s\le m$, and the symmetry $d_s=d_{2m-s}$ gives the result for all
coefficients.
\end{proof}

\section{Applications}\label{sec:applications}

We first apply Theorem~\ref{thm:node-positivity} to the nodes
appearing in Proposition~\ref{prop:unified-local-expansion}.

\begin{Theorem}\label{thm:Q-magic}
For every $d\ge2$, all magic coefficients of $L_{Q_d}(k)$ are strictly
positive.
\end{Theorem}

\begin{proof}
Write
\[
d=2m+\varepsilon,
\qquad
\varepsilon\in\{0,1\},
\qquad
\alpha=\frac{1+\varepsilon}{2}.
\]
Apply Theorem~\ref{thm:node-positivity} with
\[
X=m+\alpha,
\qquad
x_j=j+\alpha
\quad(0\le j\le m-1).
\]
Proposition~\ref{prop:unified-local-expansion} gives
\[
C_{m,\varepsilon}(u)=\Omega_{\mathbf x}(u).
\]
Hence
\[
H_{m,\varepsilon}(y)
:=
(1-y)^{2m}
C_{m,\varepsilon}\left(\frac{y}{(1-y)^2}\right)
\]
has strictly positive coefficients.  By
Proposition~\ref{prop:Q-factorization},
\[
(1-y)^{2m+\varepsilon}
L_{Q_{2m+\varepsilon}}\left(\frac{y}{1-y}\right)
=
(1+y)^\varepsilon
H_{m,\varepsilon}(y).
\]
Since $\varepsilon$ is either $0$ or $1$, the right-hand side has strictly
positive coefficients.  These are precisely the magic coefficients of
$L_{Q_d}(k)$.
\end{proof}

Since $V_{2m}^\vee=Q_{2m}$, the del Pezzo part of the classification is now
settled.  We next treat the pseudo-del Pezzo polytope.

\begin{Proposition}\label{prop:pseudo-average}
For every $m\ge1$,
\[
L_{\widetilde V_{2m}^\vee}(k)
=
\frac{(2k+1)^{2m}+L_{V_{2m}^\vee}(k)}2.
\]
\end{Proposition}

\begin{proof}
Put $d=2m$.  From the definition of $\widetilde V_d$,
\[
\widetilde V_d^\vee
=
\{x\in[-1,1]^d:x_1+\cdots+x_d\ge-1\}.
\]
Thus $L_{\widetilde V_d^\vee}(k)$ counts vectors
$z\in[-k,k]^d\cap\ZZ^d$ satisfying
\[
z_1+\cdots+z_d\ge-k.
\]
Let $T=(2k+1)^d$ and let
\[
B
=
\#\{z\in[-k,k]^d\cap\ZZ^d:
-k\le z_1+\cdots+z_d\le k\}.
\]
The central band is $kV_d^\vee$, so $B=L_{V_d^\vee}(k)$.  The involution
$z\mapsto-z$ exchanges the regions where the coordinate sum is greater
than $k$ and less than $-k$.  If their common cardinality is $C$, then
\[
T=B+2C,
\qquad
L_{\widetilde V_d^\vee}(k)=B+C,
\]
which gives the formula.
\end{proof}

\begin{Theorem}\label{thm:building-blocks}
For every $m\ge1$, all magic coefficients of
$L_{V_{2m}^\vee}(k)$ and $L_{\widetilde V_{2m}^\vee}(k)$
are strictly positive.
\end{Theorem}

\begin{proof}
Let $a_i$ and $\widetilde a_i$ denote the magic coefficients of
$L_{V_{2m}^\vee}(k)$ and $L_{\widetilde V_{2m}^\vee}(k)$,
respectively.
By Theorem~\ref{thm:Q-magic}, we have $a_i>0$ for $0\le i\le2m$.
Since $(2k+1)^{2m}=(k+(k+1))^{2m}$,
Proposition~\ref{prop:pseudo-average} gives
\[
\widetilde a_i
=
\frac{\binom{2m}{i}+a_i}{2}>0
\qquad(0\le i\le2m).
\]
\end{proof}

We now prove Theorem~\ref{thm:pseudo-symmetric-main}.
By Proposition~\ref{prop:classification}, every pseudo-symmetric
simplicial reflexive polytope $P$ is lattice equivalent to a free sum
of a centrally symmetric reflexive crosspolytope and del Pezzo and
pseudo-del Pezzo factors.
Taking polar duals turns this decomposition into a Cartesian product.
Proposition~\ref{prop:crosspolytope-magic},
Theorem~\ref{thm:building-blocks}, and Lemma~\ref{lem:closure}
therefore show that $L_{P^\vee}(k)$ is magic positive.

Suppose now that $P$ is smooth Fano.
By Proposition~\ref{prop:smooth-classification}, $P$ is a free sum of
copies of $[-1,1]$, del Pezzo polytopes, and pseudo-del Pezzo polytopes.
The two magic coefficients of
\[
L_{[-1,1]}(k)=2k+1=k+(k+1)
\]
are both equal to $1$, while
Theorem~\ref{thm:building-blocks} gives strict positivity for the other
factors.
The strict version of Lemma~\ref{lem:closure}(3) therefore shows that
all magic coefficients of $L_{P^\vee}(k)$ are strictly positive.
The lattice equivalence $\PP_{C_n}^\vee\cong Q_{n-1}$
\cite[Proposition~4.6 and its proof]{KonoikeCycles} gives
Theorem~\ref{thm:cycles} from Theorem~\ref{thm:Q-magic}.
Since the polar duals appearing in both theorems are reflexive,
Corollaries~\ref{cor:fano-consequences} and~\ref{cor:cycle-consequences}
follow from the implications of magic positivity recalled in the
Introduction.

\subsection{Further questions}

Theorem~\ref{thm:node-positivity} gives a local decomposition
\[
d_s
=
\sum_{\substack{I\subseteq S_m\\|I|=s}}
\tau_{\mathbf x}(I),
\qquad
\tau_{\mathbf x}(I)>0
\qquad (0\le s\le m).
\]
The remaining coefficients are determined by the symmetry
$d_s=d_{2m-s}$.
It would be interesting to find a direct combinatorial interpretation of
these local coefficients in the arithmetic-node specializations arising
from del Pezzo polytopes and cycles.

\begin{Question}
Do the local coefficients $\tau_{\mathbf x}(I)$ arising from
$V_{2m}^\vee$ admit a natural weighted combinatorial interpretation
related to its quadratic triangulations?
\end{Question}

\appendix

\section{A missing argument in a proof in  \cite{OhsugiTsuchiyaLocally}}
\label{app:previous-proof}

In \cite[Theorem~5.8]{OhsugiTsuchiyaLocally}, we stated that the
$h^*$-polynomial of every pseudo-symmetric simplicial reflexive polytope
is $\gamma$-positive.
The proof there uses Nill's decomposition into a centrally symmetric
reflexive crosspolytope, del Pezzo polytopes, and pseudo-del Pezzo
polytopes.
However, the centrally symmetric reflexive crosspolytope factor was
inadvertently treated as a standard crosspolytope.
We supply the missing argument here.

\begin{Proposition}\label{prop:crosspolytope-gamma}
Let $R$ be a $d$-dimensional centrally symmetric reflexive
crosspolytope, and let $C\subseteq\mathbb F_2^d$ be the even binary
code associated with $R$ as in Section~\ref{sec:crosspolytope}.
For $0\le2j\le d$, let
\[
A_{2j}
=
\#\{c\in C:\operatorname{wt}(c)=2j\}.
\]
Then
\[
h^*(R;t)
=
\sum_{j=0}^{\lfloor d/2\rfloor}
4^jA_{2j}t^j(1+t)^{d-2j}.
\]
In particular, $h^*(R;t)$ is $\gamma$-positive.
\end{Proposition}
\begin{proof}
For $x=(x_1,\ldots,x_d)\in\RR^d$, write
$\|x\|_1=|x_1|+\cdots+|x_d|$.
Using the lattice realization
$\ZZ^d\subseteq\Lambda\subseteq\frac12\ZZ^d$
from Section~\ref{sec:crosspolytope}, decompose $\Lambda$ into cosets
of $\ZZ^d$ indexed by $C$.

For an integral coordinate,
$\sum_{n\in\ZZ}t^{|n|}
=
\frac{1+t}{1-t}$,
whereas for a half-integral coordinate,
\[
\sum_{n\in\ZZ}t^{|n+1/2|}
=
\frac{2t^{1/2}}{1-t}.
\]
Hence a codeword of weight $2j$ contributes
\[
\frac{4^jt^j(1+t)^{d-2j}}{(1-t)^d}
\]
to
$\sum_{x\in\Lambda}t^{\|x\|_1}$.
Therefore
\[
\sum_{x\in\Lambda}t^{\|x\|_1}
=
\frac{1}{(1-t)^d}
\sum_{j=0}^{\lfloor d/2\rfloor}
4^jA_{2j}t^j(1+t)^{d-2j}.
\]
Since every sign vector belongs to $\Lambda^\vee$, for every
$x\in\Lambda$ the value $\|x\|_1$ is an integer.
Thus
\[
\sum_{k\ge0}L_R(k)t^k
=
\frac{1}{1-t}
\sum_{x\in\Lambda}t^{\|x\|_1}.
\]
Comparing this expression with the Ehrhart series
\[
\sum_{k\ge0}L_R(k)t^k
=
\frac{h^*(R;t)}{(1-t)^{d+1}}
\]
gives the desired formula.
\end{proof}

Together with Nill's classification and the remaining argument in
\cite[Theorem~5.8]{OhsugiTsuchiyaLocally}, Proposition~
\ref{prop:crosspolytope-gamma} completes the proof of that theorem.

\subsection*{Declaration on the use of generative AI}
OpenAI's ChatGPT was used for exploratory discussions of proof
strategies, English-language editing, and \LaTeX{} assistance.
All mathematical arguments and references were verified by the authors.

\bibliographystyle{plain}
\bibliography{bibliography}

@article {OhsugiTsuchiyaLocally,
    AUTHOR = {Ohsugi, Hidefumi and Tsuchiya, Akiyoshi},
     TITLE = {The {$h^*$}-polynomials of locally anti-blocking lattice
              polytopes and their {$\gamma$}-positivity},
   JOURNAL = {Discrete Comput. Geom.},
  FJOURNAL = {Discrete \& Computational Geometry. An International Journal
              of Mathematics and Computer Science},
    VOLUME = {66},
      YEAR = {2021},
    NUMBER = {2},
     PAGES = {701--722},
      ISSN = {0179-5376,1432-0444},
   MRCLASS = {52B20 (05C31)},
  MRNUMBER = {4292760},
MRREVIEWER = {Ruriko\ Yoshida},
       DOI = {10.1007/s00454-020-00236-6},
       URL = {https://doi.org/10.1007/s00454-020-00236-6},
}

@book {SturmfelsGBCP,
    AUTHOR = {Sturmfels, Bernd},
     TITLE = {Gr\"{o}bner bases and convex polytopes},
    SERIES = {University Lecture Series},
    VOLUME = {8},
 PUBLISHER = {American Mathematical Society, Providence, RI},
      YEAR = {1996},
     PAGES = {xii+162},
      ISBN = {0-8218-0487-1},
   MRCLASS = {13P10 (14M25 52B20)},
  MRNUMBER = {1363949},
MRREVIEWER = {P.\ Schenzel},
       DOI = {10.1090/ulect/008},
       URL = {https://doi.org/10.1090/ulect/008},
}

@book {BeckRobins,
    AUTHOR = {Beck, Matthias and Robins, Sinai},
     TITLE = {Computing the continuous discretely},
    SERIES = {Undergraduate Texts in Mathematics},
   EDITION = {Second},
      NOTE = {Integer-point enumeration in polyhedra,
              With illustrations by David Austin},
 PUBLISHER = {Springer, New York},
      YEAR = {2015},
     PAGES = {xx+285},
      ISBN = {978-1-4939-2968-9; 978-1-4939-2969-6},
   MRCLASS = {11P21 (05A15 05B15 11-02 11H06 52B05 52B20)},
  MRNUMBER = {3410115},
       DOI = {10.1007/978-1-4939-2969-6},
       URL = {https://doi.org/10.1007/978-1-4939-2969-6},
}

@article {HibiDual,
    AUTHOR = {Hibi, Takayuki},
     TITLE = {Dual polytopes of rational convex polytopes},
   JOURNAL = {Combinatorica},
  FJOURNAL = {Combinatorica. An International Journal on Combinatorics and
              the Theory of Computing},
    VOLUME = {12},
      YEAR = {1992},
    NUMBER = {2},
     PAGES = {237--240},
      ISSN = {0209-9683},
   MRCLASS = {52B20},
  MRNUMBER = {1179260},
MRREVIEWER = {P.\ McMullen},
       DOI = {10.1007/BF01204726},
       URL = {https://doi.org/10.1007/BF01204726},
}

@article {StanleyDecompositions,
    AUTHOR = {Stanley, Richard P.},
     TITLE = {Decompositions of rational convex polytopes},
   JOURNAL = {Ann. Discrete Math.},
  FJOURNAL = {Annals of Discrete Mathematics},
    VOLUME = {6},
      YEAR = {1980},
     PAGES = {333--342},
   MRCLASS = {52A43},
  MRNUMBER = {593545},
MRREVIEWER = {P.\ McMullen},
}

@book {AndrewsAskeyRoy,
    AUTHOR = {Andrews, George E. and Askey, Richard and Roy, Ranjan},
     TITLE = {Special functions},
    SERIES = {Encyclopedia of Mathematics and its Applications},
    VOLUME = {71},
 PUBLISHER = {Cambridge University Press, Cambridge},
      YEAR = {1999},
     PAGES = {xvi+664},
      ISBN = {0-521-62321-9; 0-521-78988-5},
   MRCLASS = {33-01 (33-02)},
  MRNUMBER = {1688958},
MRREVIEWER = {Bruce\ C.\ Berndt},
       DOI = {10.1017/CBO9781107325937},
       URL = {https://doi.org/10.1017/CBO9781107325937},
}

@misc{AvilaFerroniMorales,
  author        = {Avila, Nicolas and Ferroni, Luis and Morales, Alejandro H.},
  title         = {Luck and magic for {Pitman-Stanley} polytopes and parking functions},
  year          = {2026, arXiv:2603.19194},
  eprint        = {2603.19194},
  archiveprefix = {arXiv},
  primaryclass  = {math.CO},
  doi           = {10.48550/arXiv.2603.19194},
  url           = {https://arxiv.org/abs/2603.19194}
}

@article {BatyrevSelivanova,
    AUTHOR = {Batyrev, Victor V. and Selivanova, Elena N.},
     TITLE = {Einstein-{K}\"ahler metrics on symmetric toric {F}ano
              manifolds},
   JOURNAL = {J. Reine Angew. Math.},
  FJOURNAL = {Journal f\"ur die Reine und Angewandte Mathematik. [Crelle's
              Journal]},
    VOLUME = {512},
      YEAR = {1999},
     PAGES = {225--236},
      ISSN = {0075-4102,1435-5345},
   MRCLASS = {32Q20 (14J45 14M25)},
  MRNUMBER = {1703080},
MRREVIEWER = {Jaros\l aw\ A.\ Wi\'sniewski},
       DOI = {10.1515/crll.1999.054},
       URL = {https://doi.org/10.1515/crll.1999.054},
}

@article {BrandenPF,
    AUTHOR = {Br\"and\'en, Petter},
     TITLE = {On linear transformations preserving the {P}\'olya frequency
              property},
   JOURNAL = {Trans. Amer. Math. Soc.},
  FJOURNAL = {Transactions of the American Mathematical Society},
    VOLUME = {358},
      YEAR = {2006},
    NUMBER = {8},
     PAGES = {3697--3716},
      ISSN = {0002-9947,1088-6850},
   MRCLASS = {05A15 (05A05 05A19 20F55 26C10)},
  MRNUMBER = {2218995},
MRREVIEWER = {Toshihiro\ Watanabe},
       DOI = {10.1090/S0002-9947-06-03856-6},
       URL = {https://doi.org/10.1090/S0002-9947-06-03856-6},
}

@book {CoxLittleSchenck,
    AUTHOR = {Cox, David A. and Little, John B. and Schenck, Henry K.},
     TITLE = {Toric varieties},
    SERIES = {Graduate Studies in Mathematics},
    VOLUME = {124},
 PUBLISHER = {American Mathematical Society, Providence, RI},
      YEAR = {2011},
     PAGES = {xxiv+841},
      ISBN = {978-0-8218-4819-7},
   MRCLASS = {14M25 (05A15 05E45 52B12)},
  MRNUMBER = {2810322},
MRREVIEWER = {Ivan\ Arzhantsev},
       DOI = {10.1090/gsm/124},
       URL = {https://doi.org/10.1090/gsm/124},
}

@article {deBoorDividedDifferences,
    AUTHOR = {de Boor, Carl},
     TITLE = {Divided differences},
   JOURNAL = {Surv. Approx. Theory},
  FJOURNAL = {Surveys in Approximation Theory},
    VOLUME = {1},
      YEAR = {2005},
     PAGES = {46--69},
      ISSN = {1555-578X},
   MRCLASS = {41A05 (41-02)},
  MRNUMBER = {2221566},
MRREVIEWER = {Peter\ Alfeld},
}

@article {DubeauHermite,
    AUTHOR = {Dubeau, Fran{\c{c}}ois},
     TITLE = {On {H}ermite interpolation and divided differences},
   JOURNAL = {Surv. Math. Appl.},
  FJOURNAL = {Surveys in Mathematics and its Applications},
    VOLUME = {15},
      YEAR = {2020},
     PAGES = {257--279},
      ISSN = {1843-7265,1842-6298},
   MRCLASS = {41A05 (65D05)},
  MRNUMBER = {4099738},
MRREVIEWER = {Antonio\ L\'opez-Carmona},
     URL = {https://www.utgjiu.ro/math/sma/v15/a15_08.html},
}

@article {Ewald,
    AUTHOR = {Ewald, G\"unter},
     TITLE = {On the classification of toric {F}ano varieties},
   JOURNAL = {Discrete Comput. Geom.},
  FJOURNAL = {Discrete \& Computational Geometry. An International Journal
              of Mathematics and Computer Science},
    VOLUME = {3},
      YEAR = {1988},
    NUMBER = {1},
     PAGES = {49--54},
      ISSN = {0179-5376,1432-0444},
   MRCLASS = {14L32 (14J40)},
  MRNUMBER = {918178},
MRREVIEWER = {I.\ Dolgachev},
       DOI = {10.1007/BF02187895},
       URL = {https://doi.org/10.1007/BF02187895},
}

@book {EwaldBook,
    AUTHOR = {Ewald, G\"unter},
     TITLE = {Combinatorial convexity and algebraic geometry},
    SERIES = {Graduate Texts in Mathematics},
    VOLUME = {168},
 PUBLISHER = {Springer-Verlag, New York},
      YEAR = {1996},
     PAGES = {xiv+372},
      ISBN = {0-387-94755-8},
   MRCLASS = {52B20 (14M25 52B05)},
  MRNUMBER = {1418400},
MRREVIEWER = {P.\ McMullen},
       DOI = {10.1007/978-1-4612-4044-0},
       URL = {https://doi.org/10.1007/978-1-4612-4044-0},
}

@incollection {FeinsilverKocik,
    AUTHOR = {Feinsilver, Philip and Kocik, Jerzy},
     TITLE = {Krawtchouk polynomials and {K}rawtchouk matrices},
 BOOKTITLE = {Recent advances in applied probability},
     PAGES = {115--141},
 PUBLISHER = {Springer, New York},
      YEAR = {2005},
      ISBN = {0-387-23378-4},
   MRCLASS = {33D45 (05B20)},
  MRNUMBER = {2102950},
       DOI = {10.1007/0-387-23394-6\_5},
       URL = {https://doi.org/10.1007/0-387-23394-6_5},
}

@misc{FerroniShenanigans,
  author        = {Ferroni, Luis},
  title         = {Unimodality shenanigans in {Ehrhart} theory},
  year          = {2026, arXiv:2609.10513},
  eprint        = {2609.10513},
  archiveprefix = {arXiv},
  primaryclass  = {math.CO},
  doi           = {10.48550/arXiv.2609.10513},
  url           = {https://arxiv.org/abs/2609.10513}
}

@article {FerroniHigashitani,
    AUTHOR = {Ferroni, Luis and Higashitani, Akihiro},
     TITLE = {Examples and counterexamples in {Ehrhart} theory},
   JOURNAL = {EMS Surv. Math. Sci. },
  FJOURNAL = {EMS Surveys in Mathematical Sciences},
    VOLUME = {},
      YEAR = {2024},
    NUMBER = {},
     PAGES = {},
      ISSN = {},
   MRCLASS = {},
  MRNUMBER = {},
MRREVIEWER = {},
       DOI = {10.4171/EMSS/86
},
       URL = {https://ems.press/journals/emss/articles/14298122},
         note    = {Published online first}
}

@article {Gal,
    AUTHOR = {Gal, \'Swiatos\l aw R.},
     TITLE = {Real root conjecture fails for five- and higher-dimensional
              spheres},
   JOURNAL = {Discrete Comput. Geom.},
  FJOURNAL = {Discrete \& Computational Geometry. An International Journal
              of Mathematics and Computer Science},
    VOLUME = {34},
      YEAR = {2005},
    NUMBER = {2},
     PAGES = {269--284},
      ISSN = {0179-5376,1432-0444},
   MRCLASS = {52B70 (05E25 06A07)},
  MRNUMBER = {2155722},
MRREVIEWER = {Wolfgang\ K\"uhnel},
       DOI = {10.1007/s00454-005-1171-5},
       URL = {https://doi.org/10.1007/s00454-005-1171-5},
}

@phdthesis{HaaseThesis,
   author       = {Haase, Christian},
  title        = {Lattice Polytopes and Triangulations: With Applications to Toric Geometry},
  school       = {Technische Universit{\"a}t Berlin},
  year         = {2000},
  url          = {http://kobv.de}
}

@article {KonoikeCycles,
    AUTHOR = {Konoike, Masato},
     TITLE = {A new class of magic positive {E}hrhart polynomials of
              reflexive polytopes},
   JOURNAL = {S\'em. Lothar. Combin.},
  FJOURNAL = {S\'eminaire Lotharingien de Combinatoire},
    VOLUME = {93B},
      YEAR = {2025},
     PAGES = {Art. 148, 11},
      ISSN = {1286-4889},
   MRCLASS = {52B20 (05A15)},
  MRNUMBER = {4975683},
}

@misc{LiuZhangPartialPermutohedra,
  author        = {Liu, Feihu and Zhang, Zihao},
  title         = {Magic Positivity for the {Ehrhart} Polynomials of Partial Permutohedra},
  year          = {2026, arXiv:2607.03854},
  eprint        = {2607.03854},
  archiveprefix = {arXiv},
  primaryclass  = {math.CO},
  doi           = {10.48550/arXiv.2607.03854},
  url           = {https://arxiv.org/abs/2607.03854}
}

@article {LorenzNill,
    AUTHOR = {Lorenz, Benjamin and Nill, Benjamin},
     TITLE = {On smooth {G}orenstein polytopes},
   JOURNAL = {Tohoku Math. J. (2)},
  FJOURNAL = {The Tohoku Mathematical Journal. Second Series},
    VOLUME = {67},
      YEAR = {2015},
    NUMBER = {4},
     PAGES = {513--530},
      ISSN = {0040-8735,2186-585X},
   MRCLASS = {52B20 (14J45 14M25)},
  MRNUMBER = {3436539},
MRREVIEWER = {Mateusz\ Micha\l ek},
       DOI = {10.2748/tmj/1450798070},
       URL = {https://doi.org/10.2748/tmj/1450798070},
}

@article {MatsuiEtAlGraphs,
    AUTHOR = {Matsui, Tetsushi and Higashitani, Akihiro and Nagazawa, Yuuki
              and Ohsugi, Hidefumi and Hibi, Takayuki},
     TITLE = {Roots of {E}hrhart polynomials arising from graphs},
   JOURNAL = {J. Algebraic Combin.},
  FJOURNAL = {Journal of Algebraic Combinatorics. An International Journal},
    VOLUME = {34},
      YEAR = {2011},
    NUMBER = {4},
     PAGES = {721--749},
      ISSN = {0925-9899,1572-9192},
   MRCLASS = {52B12 (05C31)},
  MRNUMBER = {2842918},
MRREVIEWER = {Mihai\ Cipu},
       DOI = {10.1007/s10801-011-0290-8},
       URL = {https://doi.org/10.1007/s10801-011-0290-8},
}

@incollection {NillPseudoSymmetric,
    AUTHOR = {Nill, Benjamin},
     TITLE = {Classification of pseudo-symmetric simplicial reflexive
              polytopes},
 BOOKTITLE = {Algebraic and geometric combinatorics},
    SERIES = {Contemp. Math.},
    VOLUME = {423},
     PAGES = {269--282},
 PUBLISHER = {Amer. Math. Soc., Providence, RI},
      YEAR = {2006},
      ISBN = {978-0-8218-4080-1; 0-8218-4080-0},
   MRCLASS = {52B20 (14M25 52B12)},
  MRNUMBER = {2298762},
MRREVIEWER = {Seth\ Sullivant},
       DOI = {10.1090/conm/423/08082},
       URL = {https://doi.org/10.1090/conm/423/08082},
}

@article {SongAlpha,
    AUTHOR = {Song, Jian},
     TITLE = {The {$\alpha$}-invariant on toric {F}ano manifolds},
   JOURNAL = {Amer. J. Math.},
  FJOURNAL = {American Journal of Mathematics},
    VOLUME = {127},
      YEAR = {2005},
    NUMBER = {6},
     PAGES = {1247--1259},
      ISSN = {0002-9327,1080-6377},
   MRCLASS = {32Q20 (14J45 14M25)},
  MRNUMBER = {2183524},
MRREVIEWER = {Jaros\l aw\ A.\ Wi\'sniewski},
       URL =
              {http://muse.jhu.edu/journals/american_journal_of_mathematics/v127/127.6song.pdf},
}
\end{document}